\documentclass[journal,web]{ieeecolor}
\usepackage{generic}
\usepackage{cite}
\usepackage{graphicx}
\usepackage{textcomp}
\usepackage{dsfont}
\usepackage{enumerate}
\usepackage{lscape}
\usepackage{booktabs}

\usepackage{times} % assumes new font selection scheme installed
\usepackage{amsmath} % assumes amsmath package installed
\usepackage{amssymb}  % assumes amsmath package installed
\usepackage{amsfonts}  % assumes amsmath package installed

\usepackage{amsthm}

\newtheoremstyle{mytheorem}
  {3pt}              % Space above
  {3pt}              % Space below
  {\itshape}      % Body font
  {}                 % Indent
  {\bfseries}        % Theorem heading font
  {:}                % Punctuation after theorem number
  {.5em}             % Space after heading
  {}                 % Heading specification

\theoremstyle{mytheorem}
\newtheorem{theorem}{Theorem}
\newtheorem{proposition}{Proposition}
\newtheorem{lemma}{Lemma}

\newtheorem{remark}{Remark}

\DeclareMathOperator{\Span}{span}
\DeclareMathOperator{\Grad}{grad}
\DeclareMathOperator{\Hess}{Hess}
\DeclareMathOperator{\Div}{div}

\newcommand{\volg}{\mathrm{vol}}
\newcommand{\M}{\mathcal{M}}
\newcommand{\Pcal}{\mathcal{P}(\M)}
\newcommand{\Fcal}{\mathcal{F}}
\newcommand{\Wcal}{\mathcal{W}}
\newcommand{\dTV}[2]{\left\|#1-#2\right\|_{\mathrm{TV}}}
\newcommand{\dist}{d}

\begin{document}
\title{Stochastic Lie-Bracket Approximations for Zeroth-Order Optimization on Manifolds}
\author{Mahmoud Abdelgalil, Miroslav Krstic, and Jorge I. Poveda
\thanks{This work was supported in part by the Defense Advanced Research Projects Agency (DARPA) under Grant No. HR00112530225.}
\thanks{
M. Abdelgalil is with the Department of Mechanical and Aerospace Engineering, University at Buffalo, State University of New York, Amherst, NY, 14260, USA (email:  maabdelg@buffalo.edu). Miroslav Krstic is with the Department of Mechanical and Aerospace Engineering, University of California San Diego, La Jolla, CA 92093, USA (email: mkrstic@ucsd.edu). J. I. Poveda is with the Department of Electrical and Computer Engineering,
        University of California San Diego, La Jolla, CA 92093, USA (email: jipoveda@ucsd.edu).
}
}

\maketitle

\pagestyle{empty}
\thispagestyle{empty}
\begin{abstract}
We propose a continuous-time algorithm for zeroth-order global optimization of smooth functions on compact manifolds that
synthesizes Riemannian gradient dynamics using only objective measurements, without evaluating gradients. The algorithm is
formulated as an ordinary differential equation driven by piecewise-smooth McShane-type approximations of the Wiener
process, with a geometric compensation term that depends only on the actuation frame and not on the objective function.
Our main result establishes infinite-horizon practical probabilistic guarantees for fixed approximation parameters.
To this end, we show that the algorithm's trajectories approximate, in the mean-square sense, a Riemannian Langevin
diffusion whose unique invariant Gibbs measure concentrates on the set of global minimizers. We exploit this approximation
to establish finite-horizon probabilistic convergence guarantees and directly analyze the algorithm's long-time behavior to obtain the infinite-horizon result.
\end{abstract}

\begin{IEEEkeywords}
Stochastic optimization, Zeroth-Order Optimization, Optimization on Manifolds.
\end{IEEEkeywords}

\section{INTRODUCTION}
\label{sec:introduction}

 \IEEEPARstart{M}{odel-free} optimization algorithms that rely only on real-time measurements of an objective function play an important role in adaptive control, LLM fine-tuning, and learning-based decision-making systems. A key challenge is establishing global convergence guarantees without restrictive convexity-like assumptions, particularly in the presence of complex constraints induced by smooth manifolds. In such settings, it is well-known that spurious equilibria may arise, precluding global convergence guarantees for gradient-based algorithms \cite{ochoa2025robust}. A standard approach to overcoming spurious equilibria in nonconvex Euclidean optimization is to inject persistent stochasticity, transforming the gradient flows into a Langevin diffusion whose invariant Gibbs measure concentrates near global minimizers. This principle underlies discrete-time \cite{gelfand1991recursive} and continuous-time simulated annealing \cite{chiang1987diffusion,gelfand1991recursive,holley1989asymptotics} and stochastic approximation-based algorithms \cite{borkar2008stochastic,crisafulli2026two,maass2022tracking,li2024zeroth}, and has also motivated Langevin-based methods on Riemannian manifolds \cite{girolami2011riemann}. However, first-order methods require gradient information, and existing stochastic zeroth-order methods are predominantly discrete-time, with
constraints typically enforced through projections at each iteration. Consequently, continuous-time geometric frameworks for global
zeroth-order optimization on smooth manifolds remain largely unexplored.
 
On the other hand, a prominent dynamical-systems approach to zeroth-order optimization is
extremum seeking (ES) control \cite{krstic2000stability}, which combines
persistent probing and averaging to establish control-theoretic stability
guarantees. Lie-bracket-based formulations of ES \cite{durr2013lie}
have enabled systematic designs with bounded update rates
\cite{scheinker2014extremum}, local asymptotic stability
\cite{grushkovskaya2018class}, and global practical stability
\cite{abdelgalil2024initialization}, with extensions to manifolds
\cite{durr2014extremum,grushkovskaya2018class,suttner2022extremum,
abdelgalil2023singularly} and hybrid systems
\cite{abdelgalil2025event,abdelgalil2025lie,abdelgalil2026hybrid}.
However, smooth deterministic ES on manifolds inherits a fundamental
limitation of the gradient flows it approximates: convergence is
generally local or, at best, almost global. Although hybrid ES based
on synergistic switching can overcome this limitation
\cite{abdelgalil2025lie,ochoa2025robust}, it typically requires
additional knowledge of the cost function to escape undesirable
equilibria.

Motivated by this background, in this paper we introduce a new class of continuous-time zeroth-order optimization algorithms based on \emph{stochastic Lie-bracket approximations} (SLBA) that provide global convergence guarantees using only objective measurements, without requiring additional knowledge of the cost function. The proposed dynamics combine the
bounded-update-rate design of \cite{scheinker2014extremum} with Brownian, rather than periodic, dithering. The algorithm is modeled by a \emph{random ODE} \cite{han2017random} on a smooth
compact manifold $\M$, driven by piecewise-smooth McShane-type
approximations of a Wiener process \cite{mcshane1975stochastic}. We show that, in the Wong--Zakai limit
\cite{wong1965convergence,ikeda1981stochastic,sussmann1991limits,
friz2010multidimensional}, its trajectories converge to those of a
Stratonovich SDE whose generator, under a suitable drift compensation,
coincides with that of a Riemannian Langevin diffusion, despite the
algorithm using only zeroth-order information. We also show that the resulting Gibbs
invariant measure concentrates on the global minimizers of the objective
function, and we quantify this concentration and establish finite-horizon practical guarantees through mixing properties of the limiting SDE.

We further study the algorithm itself for a fixed discretization parameter
$\varepsilon>0$. In this scenario, its sampled dynamics define a Markov chain for which we
establish the existence of invariant measures and characterize their
long-time behavior, showing weak convergence toward the Gibbs measure as
$\varepsilon\to0$. This yields global practical probabilistic optimization
guarantees over infinite horizons and from arbitrary initial conditions.
Compared with deterministic Lie-bracket ES on manifolds
\cite{durr2014extremum}, Brownian dithering replaces local practical
convergence with global measure-theoretic guarantees. Relative to stochastic ES \cite{liu2012stochastic,suttner2026averaging}, the main novelties are the incorporation of the manifold setting via Lie brackets
and, importantly, the invariant-measure analysis of the approximating
dynamics at fixed $\varepsilon$. In particular, our result is the first in the literature on stochastic ES to show infinite-time probabilistic guarantees for fixed $\varepsilon>0$. This Markov-chain perspective is related to the analysis of numerical SDE schemes
\cite{mattingly2002ergodicity,meyn2012markov}, but our algorithm is
instead a continuous-time random ODE.

\vspace{-0.2cm}
\section{NOTATION AND PRELIMINARIES}\label{sec:notation}
We denote the natural numbers by $\mathbb N_0:=\{0,1,2,\dots\}$ and $\lceil a\rceil$ is the least integer greater than or equal to $a$. For a
subset $A$ of a metric space, $\overline A$ and $\mathrm{int}(A)$ denote closure and interior, $\mathbf 1_A$ the
indicator function, and, for $\eta>0$ let $A^\eta:=\{y:\ \dist(y,A)<\eta\}$ be the open $\eta$-neighbourhood of $A$. The Lebesgue measure on $\mathbb R^N$ is written $\mathrm{Leb}$, and
$B_\rho(a)\subset\mathbb R^N$ is the open Euclidean ball of radius $\rho$ centred at $a$. The Euclidean norm is denoted by $|\cdot|$ and $\|\cdot\|$ the induced operator norm. For a linear map $L$, $\sigma_{\min}(L)$ is its smallest
singular value. We use $\M$ to denote a smooth, connected, compact, orientable $n$-dimensional manifold without boundary, embedded in $\mathbb R^m$, $T_x\M$ and $T^*_x\M$ are its tangent and cotangent spaces at $x$, respectively, and
$\mathcal B(\M)$ denotes its Borel $\sigma$-algebra. We use $C^\infty(\M)$, $C(\M)$, and $\mathfrak X(\M)$ for smooth
functions, continuous functions, and smooth vector fields on $\M$, respectively, and $\Span_{\mathbb R}$ for the linear span with constant real coefficients.
For $f,h\in\mathfrak X(\M)$ and $\psi\in C^\infty(\M)$, $f(\psi)$ denotes the derivative of $\psi$ along $f$,
$[f,h]$ the Lie bracket, $f\otimes h$ is the tensor product, and $\Phi^f_s$ is the time-$s$ flow of $f$. Differentials of
maps are written $\mathrm D$, and are regarded as maps into the ambient $\mathbb R^m$ whenever differentials at
distinct base points are compared. Given a Riemannian metric $\langle\cdot,\cdot\rangle_g$ on $\M$, the symbols $\Grad$, $\Div$, $\Delta$, $\Hess$, $\nabla$, and $\volg$ denote the gradient, divergence,
Laplace--Beltrami operator, Hessian, Levi-Civita connection, and Riemannian density, respectively, and
$\mathrm{Vol}(\M):=\int_\M\volg$. We write $\dist(\cdot,\cdot)$ for the geodesic distance and $B(z,\rho)\subset\M$ for the open
geodesic ball. Since $\M$ is compact, $\dist$ and the ambient distance $|x-y|$ are equivalent, modulo constants. We use $\|\psi\|_\infty:=\sup_\M|\psi|$ to denote the supremum norm, and $\mathrm{Lip}(\psi)$ to denote the smallest constant $L$ with $
|\psi(x)-\psi(y)|\le L\,\dist(x,y)
$. For a positive measure $\nu$ on $\M$, we equip the space $L^p(\nu)$ with the norm
$\|\psi\|_{L^p(\nu)}:=(\int_\M|\psi|^p\,\mathrm d\nu)^{1/p}$. We use $\Pcal$ to denote the Borel probability measures on $\M$ with the topology of weak convergence, and $\rightharpoonup$ to denote weak convergence. For $\nu,\nu'\in\Pcal$, we use $\dTV{\nu}{\nu'}$ to denote the total variation distance
and
\begin{equation}\label{dbldef}
d_{\mathrm{BL}}(\nu,\nu'):=\sup\Big\{\textstyle\int_\M f\,\mathrm d(\nu-\nu'): \|f\|_\infty\le1,
\mathrm{Lip}(f)\le1\Big\},
\end{equation}
to denote the bounded-Lipschitz distance, which metrizes weak convergence on the compact space $\M$. We use $\mathcal N(0,\Sigma)$ to denote the centered Gaussian distribution with covariance $\Sigma$. %For a Markov transition kernel $P:\M\times\M\rightarrow[0,1]$ on $\M$, we write
%$P\psi(x):=\int_\M\psi(y)\,P(x,\mathrm dy)$, and $P^k$ for the $k$-step kernel. 
Finally, we use $\mathbb E$ and $\mathbb P$ to denote expectation and probability.

\section{PROBLEM FORMULATION AND THE SLBA ALGORITHM}
\label{sec:formulation}

Let $\phi\in C^\infty(\M)$ be the objective function, accessible only through evaluations $\phi(x)$. Our goal is to solve the problem
\begin{align}\label{eq:problem}
    \min_{x\in\M}~\phi(x).
\end{align}
Let $\phi_\star:=\min_{x\in\mathcal M}\phi(x)$ and define the set of minimizers
$$
S_\star:=\{x\in\mathcal M:\ \phi(x)=\phi_\star\}.
$$ 
To solve problem \eqref{eq:problem}, we will design an algorithm expressed as a \emph{random ODE} \cite{han2017random}, whose trajectories converge to the set $S_\star$. To that end, let $\Fcal=\{f_1,\dots,f_r\}\subset\mathfrak{X}(\M)$ be a family of vector fields satisfying the span condition
\begin{equation}\label{eq:span}
\Span_{\mathbb{R}}\{f_1(x),\dots,f_r(x)\}=T_x\M,\qquad \forall x\in\M.
\end{equation}
Such a family always exists, e.g., the tangential projections of the ambient coordinate vector fields of the embedding $\M\subset\mathbb{R}^m$ provide $r=m$ fields satisfying \eqref{eq:span}. Let $\Wcal:=\{W\in C([0,\infty);\mathbb{R}^{2r}): W_0=0\}$ be the standard Wiener space, i.e., the space of continuous functions equipped with the Wiener measure, and, for any $W\in \Wcal$, let 
$$
w_k^\ell:=W^\ell_{(k+1)\varepsilon}-W^\ell_{k\varepsilon}
$$ 
denote the increments over the grid of step $\varepsilon>0$, with $k\in\mathbb{N}_0$, and $\ell\in\{1,\dots,2r\}$. For $\kappa>0$, define the functions 
\begin{align}
U^1_\kappa(s)&:=s+\frac{1+\sqrt{2\pi^2\kappa+1}}{2\pi}\sin(2\pi s),\label{eq:U1}\\
U^2_\kappa(s)&:=s+\frac{1+\sqrt{2\pi^2\kappa+1}}{\pi}\sin^2(\pi s),\label{eq:U2}
\end{align}
and, for $t\in[k\varepsilon,(k+1)\varepsilon)$, $i\in\{1,\dots,r\}$, $j\in\{1,2\}$, the piecewise-smooth process
\begin{equation}\label{eq:B}
B^{2i+j-2}_\varepsilon(t,W):=W^{2i+j-2}_{k\varepsilon}+{w}_k^{2i+j-2}\,U^{\iota_i^j({w}_k)}_\kappa\!\Big(\tfrac{t-k\varepsilon}{\varepsilon}\Big),
\end{equation}
where the random index $\iota_i^j({w})$ is defined by
\begin{align*}
    \iota_i^j({w}):=
    \begin{cases}
        j & {w}^{2i-1}{w}^{2i}\ge0 \\
        3-j & {w}^{2i-1}{w}^{2i}<0
    \end{cases}
\end{align*}
i.e., the two functions are swapped within the $i$-th pair when the corresponding increments have opposite signs. This piecewise smooth interpolation of Brownian motion is known as a McShane-type interpolation \cite{mcshane1975stochastic}. By construction, $B^\ell_\varepsilon$ is piecewise differentiable, and it has a well-defined piece-wise-continuous derivative $b^\ell_\varepsilon:=\dot B^\ell_\varepsilon$. The Stochastic Lie-Bracket Approximation (SLBA) algorithm is then defined by the following random ODE:
\begin{equation}\label{eq:rode}
\dot x = \bar f_0(x)+\frac{1}{\sqrt\kappa}\sum_{i=1}^r\sum_{j=1}^2 F_j(\phi(x))\,f_i(x)\, b^{2i+j-2}_\varepsilon(t,W),
\end{equation}
with $F_1=\cos$, $F_2=-\sin$, and a compensation drift $\bar f_0\in\mathfrak{X}(\M)$, specified later by Theorem \ref{thm:generator} in terms of the family of vector fields $\mathcal{F}$ only, and independently from the function $\phi$. Note that \eqref{eq:rode} is an ODE for each sample path $W$, and the objective enters only through the measured values $\phi(x(t))$. For each initial condition $x_0\in\mathcal{M}$ and each sample path $W$, we write $x^\varepsilon_t(x_0)$ to denote the solution of \eqref{eq:rode}. By continuity, solutions exist globally, and remain on $\M$ since the right-hand side is tangent to $\M$.

By the Wong--Zakai theory for approximations with corrections \cite{mcshane1975stochastic}, \cite[Ch.~7]{ikeda1981stochastic}, the trajectories of \eqref{eq:rode} approximate, as $\varepsilon\to0$, the sample paths of the following Stratonovich SDE:
\begin{multline}\label{eq:sde}
\mathrm{d}\bar x_t=\Big(\bar f_0(\bar x)+\sum_{i=1}^r\big[F_1(\phi)f_i,\,F_2(\phi)f_i\big](\bar x)\Big)\mathrm{d}t\\
+\tfrac{1}{\sqrt\kappa}\sum_{i=1}^r\sum_{j=1}^2F_j(\phi(\bar x))f_i(\bar x)\bullet\mathrm{d}W^{2i+j-2}_t,
\end{multline}
where the bracket term is the drift correction generated by \eqref{eq:U1}--\eqref{eq:B}. We write $\bar x_t(x_0)$ for the solution of \eqref{eq:sde} and $\bar P_t$ for its Markov semigroup. The following theorem, adapted from the results of \cite[Thm.~7.2 and Rem.~7.2]{ikeda1981stochastic}, establishes a \emph{stochastic convergence of trajectories} property that will be instrumental for our result:

\begin{theorem}\label{thm:scot}
For every $\delta,T>0$, there exists $\varepsilon^*(\delta,T)>0$ such that for all $\varepsilon\in(0,\varepsilon^*)$,
\begin{equation}\label{eq:wz}
\sup_{x_0\in\M}\ \mathbb{E}\Big[\sup_{t\in[0,T]}\big\|x^\varepsilon_t(x_0)-\bar x_t(x_0)\big\|^2\Big]<\delta,
\end{equation}
when both solutions are driven by the same sample path. \hfill $\square$
\end{theorem}
In what remains of this manuscript, we utilize the stochastic convergence of trajectories property established by Theorem \ref{thm:scot}, in conjunction with the properties of the limiting SDE \eqref{eq:sde}, to analyze the behavior of the random ODE \eqref{eq:rode}.

\section{THE LIMITING DIFFUSION AND ITS GIBBS MEASURE}
\label{sec:sde}

We begin by analyzing the limiting SDE \eqref{eq:sde}. The span condition \eqref{eq:span} induces a canonical metric on $\M$. Indeed, since the co-metric $\sum_{i=1}^r f_i\otimes f_i$ is positive definite by virtue of \eqref{eq:span}, there exists a unique Riemannian metric $\langle\cdot,\cdot\rangle$ on $\M$ such that, $\forall\psi\in C^\infty(\M)$, 
\begin{equation}\label{eq:metric}
\Grad\psi=\sum\nolimits_{i=1}^rf_i(\psi)\,f_i.
\end{equation}
In the sequel, all geometric operators refer to this metric.

\begin{theorem}\label{thm:generator}
Let $\langle\cdot,\cdot\rangle$ be the metric \eqref{eq:metric} and define
$$\bar f_0=-\tfrac{1}{2\kappa}\sum_{i=1}^r\nabla_{f_i}f_i.$$
Then, the following holds:
\begin{enumerate}[(a)]
\item The infinitesimal generator of the SDE \eqref{eq:sde} is given by
\begin{equation}\label{eq:generator}
\mathcal{L}\psi=\tfrac{1}{2\kappa}\Delta \psi-\langle\Grad\phi,\Grad\psi\rangle_g,
\end{equation}
for all $\psi\in C^\infty(\M)$. 
\item The SDE \eqref{eq:sde} admits the unique invariant measure
\begin{equation}\label{eq:gibbs}
\mu_\kappa=c_\kappa^{-1}e^{-2\kappa\phi}\,\volg,
\end{equation}
where $c_\kappa:=\int_\M e^{-2\kappa\phi}\volg$. \hfill $\square$
\end{enumerate}
\end{theorem}

\begin{remark}
A key implication of Theorem \ref{thm:generator} is that the infinitesimal generator of the SDE \eqref{eq:sde} coincides with the generator of the \emph{Riemannian Langevin SDE} defined by 
\begin{equation}\label{LSDE}
\mathrm{d}x=-\Grad\phi(x)\mathrm{d}t+\kappa^{-1/2} dB^g,
\end{equation}
where $B^g$ is the Brownian motion on $\M$, with respect to the metric $g$, whose generator is $\Delta$ \cite[Chapter V]{ikeda1981stochastic}. This relationship plays a key role in our analysis. \hfill $\square$
\end{remark}

\vspace{0.1cm}
The relationship between the SDE \eqref{eq:sde} and the Langevin dynamics \eqref{LSDE} establishes \eqref{eq:sde} as a suitable optimization algorithm. In particular, by leveraging the explicit form of the invariant measure $\mu_\kappa$, provided by Theorem \ref{thm:generator}, we can study its asymptotic behavior when $\kappa \gg 1$. To that end, we have the following lemma. We recall that $S_\star$ denotes the set of minimizers of $\phi$.

\begin{lemma}\label{lem:conc}
For every closed $K\subseteq\M$ with $K\cap S_\star=\emptyset$ there exist $C_K,c_K>0$, such that
\begin{equation}\label{eq:conc}
\mu_\kappa(K)\ \le\ C_K\,e^{-2\kappa c_K},
\end{equation}
for all $\kappa>0$. In particular, $\mu_\kappa(U)\to1$ as $\kappa\to\infty$ for every open $U\supset S_\star$. \hfill $\square$
\end{lemma}

\vspace{0.1cm}
Although the invariant measure $\mu_\kappa$ concentrates on the set of minimizers $S_\star$ in the limit $\kappa\rightarrow \infty$, the convergence of the marginals of the process \eqref{eq:sde} to the invariant measure $\mu_\kappa$ is, in general, slow for large $\kappa$. This is due to the fact that, for a general compact boundary-less manifold $\mathcal{M}$ and a general function $\phi$, the spectral gap of the operator $\mathcal{L}$ may decay in the limit $\kappa\rightarrow \infty$. Nevertheless, for a fixed $\kappa>0$, the spectral gap of $\mathcal{L}^*$ is uniformly bounded away from zero and, as a result, convergence is guaranteed. On the other hand, there is an important case wherein the spectral gap of $\mathcal{L}^*$ can be uniformly upper bounded away from $0$ in the limit $\kappa\rightarrow\infty$. The following lemma formalizes these claims.
\vspace{0.1cm}
\begin{lemma}\label{lem:convergence}
    For any $\kappa>0$, let $\rho_\kappa$ be the density of the invariant measure $\mu_\kappa$ given by \eqref{eq:gibbs}. Then, the following holds:
    \begin{enumerate}[(a)]
    \item There exists $\lambda_\kappa > 0$ such that for any initial probability density $\rho_0\in C^\infty(\mathcal{M})$ the solution $\rho_t$ to the Fokker-Planck equation
    \begin{align}
        \partial_t\rho_t = \mathcal{L}^*\rho_t = \Div(\rho_t \Grad\phi) + \frac{1}{2\kappa}\Delta  \rho_t,
    \end{align}
    satisfies
\begin{align}\label{eq:exp_bound}
        \|\rho_t-\rho_\kappa\|_{L^1(\volg)}&\leq \mathrm{e}^{-\lambda_\kappa t}\|\rho_0-\rho_\kappa\|_{L^2(\rho_{\kappa}^{-1}\volg)},
    \end{align}
    for all $t>0$.
    \vspace{0.1cm}
    \item If $\phi\in C^\infty(\mathcal{M})$ is a Morse function with a unique local minimum, i.e., there exists a unique critical point $x_\star\in\mathcal{M}$ of $\phi$ for which $\Hess \phi(x_\star)\succ0$, then there exists $\lambda > 0$ such that \eqref{eq:exp_bound} holds uniformly for all $\kappa > 0$. \hfill $\square$ 
    \end{enumerate}
\end{lemma}

\vspace{0.2cm}
Having characterized the invariant measure of the limiting Riemannian Langevin diffusion, we next quantify the rate at which the distribution of its solutions approaches this measure. The following result establishes exponential convergence in total variation, with a rate determined by the spectral gap of the infinitesimal generator.
\begin{lemma}\label{lem:conv_sde_FK}
    For any $\kappa>0$, let $\lambda_\kappa>0$ be the spectral gap of $\mathcal{L}$. Let $\bar{x}_t$ be a solution of the SDE \eqref{eq:sde} such that the initial condition $\bar{x}_0\sim \rho_o\,\volg$, for any smooth probability distribution $\rho_0$. Then, for every measurable set $U\subset\mathcal{M}$, we have that
    \begin{align}
        |\mathbb{P}\{\bar{x}_t\in U\}-\mu_\kappa(U)|\leq \frac{1}{2}\|\rho_0-\rho_\kappa\|_{L^2(\rho_{\kappa}^{-1}\volg)}\mathrm{e}^{-\lambda_\kappa t},
    \end{align}
    for all $t\geq 0$. In particular, if $\rho_0 = (\int_{\mathcal{M}} \volg)^{-1}$, i.e. the uniform probability density, then
    \begin{align}
        |\mathbb{P}\{\bar{x}_t\in U\}-\mu_\kappa(U)|\leq \frac{1}{2}\mathrm{e}^{\kappa\gamma_\phi-\lambda_\kappa t},
    \end{align}
    where $\gamma_\phi:=\max_{x\in\mathcal{M}}\phi(x) - \min_{x\in\mathcal{M}}\phi(x)$. \hfill $\square$
\end{lemma}

\vspace{0.1cm}
Based on the properties we have obtained so far for the limiting SDE \eqref{eq:sde}, we may now utilize the stochastic convergence-of-trajectories property established by Theorem \ref{thm:scot}. In particular, if $x^\varepsilon_t$ and $\bar{x}_t$ are the solutions for the random ODE \eqref{eq:rode} and the SDE \eqref{eq:sde}, respectively, from the same initial condition $x^\varepsilon_0=\bar{x}_0\sim\rho_0\,\volg$ and for the same driving Brownian motion path $W$, then the bound in Theorem \ref{thm:scot} implies that the probability measure of $x_t$ convergences to the probability measure of $\bar{x}_t$ in the limit $\varepsilon\rightarrow 0$, uniformly on any compact interval $[0,T]$. In particular, we have the following proposition.
\begin{proposition}\label{prop:finite_horizon} 
    For any $\kappa>0$, let $\lambda_\kappa>0$ be the spectral gap of $\mathcal{L}$. Then, for any $T>0$, any $\epsilon>0$, any $\kappa>0$, and any open set $U\subset \mathcal{M}$ with a smooth boundary, there exists $\varepsilon^*\in(0,\infty)$ such that the following bound holds
    \begin{align*}
        |\mathbb{P}\{x^\varepsilon_t\in U\}-\mu_{\kappa}(U)| < \frac{1}{2}\|\rho_0-\rho_\kappa\|_{L^2(\rho_{\kappa}^{-1}\volg)} \mathrm{e}^{-\lambda_\kappa t} + \epsilon.
    \end{align*}
    for all $\varepsilon\in(0,\varepsilon^*)$ and all $t\in[0,T]$. \hfill $\square$ 
\end{proposition}
\section{INVARIANT MEASURES OF THE ALGORITHM}
\label{sec:chain}
The exponential bound established in Proposition \ref{prop:finite_horizon} is restricted to compact time intervals and therefore does not characterize the asymptotic behavior of the system as $t\to\infty$. In this section, we remove that limitation by directly analyzing the limiting behavior of the probability measure of \eqref{eq:rode}. 
Throughout, the constants $\varepsilon,\kappa>0$ are fixed unless stated otherwise. Also,  fo a Markov transition kernel $P:\M\times\M\rightarrow[0,1]$ on $\M$, we write
$P\psi(x):=\int_\M\psi(y)\,P(x,\mathrm dy)$, and $P^k$ for the $k$-step kernel.

%\subsection{Markovianity, invariant measures, and convergence to the Gibbs measure}
%

\vspace{0.1cm}
Let $t=k\varepsilon$ and define the sampled process $X_k:=x^\varepsilon_{k\varepsilon}$, where $x^\varepsilon_{t}$ is a solution to \eqref{eq:rode}. Then, we have the following. 
\begin{lemma}\label{lem:markov}
There exists a measurable map $\Psi_\varepsilon:\M\times\mathbb{R}^{2r}\to\M$ such that the sampled process $X_k$ satisfies 
$$X_{k+1}=\Psi_\varepsilon(X_k,{w}_k),$$
where the increments ${w}_k$ are i.i.d. samples from $\mathcal{N}(0,\varepsilon I_{2r})$. That is, $(X_k)$ is a time-homogeneous Markov chain with the transition kernel
$$P_\varepsilon(x,A)=\mathbb{P}\{\Psi_\varepsilon(x,{w})\in A\}.$$
Moreover, $P_\varepsilon\psi \in C(\M)$ for any $\psi\in C(\M)$, where
$$P_\varepsilon\psi(x):=\int_{\mathbb{R}^{2r}}\psi(\Psi_\varepsilon(x,{w}))\,\gamma_\varepsilon(\mathrm{d}{w}),~~\gamma_\varepsilon:=\mathcal{N}(0,\varepsilon I_{2r}).$$
That is, $P_\varepsilon$ maps $C(\M)$ into $C(\M)$. \hfill $\square$
\end{lemma}
The Markovian characterization of the sampled dynamics in
Lemma~\ref{lem:markov} enables us to study their long-run behavior through
the associated transition kernel $P_\varepsilon$. In particular, the
Feller property established above, together with the compactness of
$\mathcal{M}$, guarantees the existence of an invariant probability
measure for every fixed $\varepsilon>0$, as formalized next.
\begin{proposition}\label{prop:existence}
For every $\varepsilon>0$, the chain $(X_k)$ admits an invariant probability measure $\pi_\varepsilon\in\Pcal$. \hfill $\square$
\end{proposition}
%
%\subsection{Convergence to the Gibbs measure}
Having established the existence of an invariant measure, it is natural to study the long-time behavior of the chain. The following result characterizes such behavior (cf. \eqref{dbldef}):
\begin{theorem}\label{thm:lawconv}
Fix $\kappa>0$ and let $\mu_\kappa$ be the unique invariant measure \eqref{eq:gibbs}. Then, for every $\delta>0$, there exist $T>0$ and $\varepsilon^\ast>0$ such that, for all
$\varepsilon\in(0,\varepsilon^\ast]$ and all integers $k\ge T/\varepsilon$, we have that
$$
\sup_{x\in\M}\ d_{\mathrm{BL}}\big(P^{k}_\varepsilon(x,\cdot),\,\mu_\kappa\big)\ \le\ \delta.
$$
In particular, every invariant probability
measure $\mu_\varepsilon$ of $P_\varepsilon$ satisfies $d_{\mathrm{BL}}(\mu_\varepsilon,\mu_\kappa)\le\delta$
for all $\varepsilon\in(0,\varepsilon^\ast]$. \hfill $\square$
\end{theorem}
Theorem \ref{thm:lawconv} establishes uniform convergence of the transition probabilities
of the approximating dynamics toward the Gibbs measure $\mu_\kappa$ in the
bounded-Lipschitz metric. We next translate this measure-theoretic
approximation into explicit probabilistic bounds relative to the set
$S_\star$ of global minimizers. In particular, the following proposition provides lower bounds on the probability of being in any neighborhood of $S_\star$,
as well as upper bounds on the probability of being in closed sets separated
from $S_\star$.
\begin{proposition}\label{prop:practical}
In the setting of Theorem~\ref{thm:lawconv}, for all $\varepsilon\in(0,\varepsilon^\ast]$ and all integers
$k\ge T\varepsilon^{-1}$, the following holds:
\begin{enumerate}
\item If $U\subsetneq\M$ is an open neighborhood of $S_\star$ and
$\eta_U:=\min\{1,\tfrac12\dist(S_\star,\M\setminus U)\}$, then
$$
\inf_{x\in\M}\ P^{k}_\varepsilon\big(x,U\big)\ \ge\ 1-C_U\,e^{-2\kappa c_U}-\delta\,\eta_U^{-1};
$$
\item If $\emptyset\ne K\subseteq\M$ is closed with $K\cap S_\star=\emptyset$ and
$\eta_K:=\min\{1,\tfrac12\dist(K,S_\star)\}$, then
$$
\sup_{x\in\M}\ P^{k}_\varepsilon\big(x,K\big)\ \le\ C_K\,e^{-2\kappa c_K}+\delta\,\eta_K^{-1},
$$
\end{enumerate}
where $C_U,c_U>0$ are the constants of \eqref{eq:conc} associated with the closed set
$\M\setminus S_\star^{\eta_U}$, and $C_K,c_K>0$ are those associated with $\overline{K^{\eta_K}}$. In particular, they depend only on $U$, respectively $K$, and are independent of $\kappa$, $\varepsilon$, $k$, and $x$. \hfill $\square$
\end{proposition}
%

    %
%\end{example}

\section{PROOFS}
\vspace{0.1cm}\noindent 
\textbf{Proof of Theorem 2:} To simplify notation, let $\phi_j = F_j \circ\phi$. Utilizing the properties of the Lie Bracket, we compute
\begin{align}
    [\phi_1 f_i,\phi_2 f_i] = (\phi_1 f_i(\phi_2) - \phi_2 f_i(\phi_1)) f_i.
\end{align}
Consequently, the limiting SDE \eqref{eq:sde} simplifies to
 \begin{align}\label{eq:SLBA-limit-sde-simplified}
     \mathrm{d}\bar{x}_t&=f_0(\bar{x})\,\mathrm{d}t + \frac{1}{\sqrt{\kappa}}\sum_{i=1}^r\sum_{j=1}^2 \phi_j f_i(\bar{x})\bullet dW_t^{2i+j-2} ,
 \end{align}
where the vector field $f_0$ is defined by
\begin{align}\label{eq:drift_defn_1}
    f_0&:= \bar{f}_0+\sum_{i=1}^r(\phi_1 f_i(\phi_2) - \phi_2 f_i(\phi_1))f_i.
\end{align}
The associated infinitesmal generator $\mathcal{L}$ is defined by
\begin{align}
    \mathcal{L}\psi = f_0(\psi) + \frac{1}{2\kappa}\sum_{i=1}^r\sum_{j=1}^2\phi_j f_i(\phi_j f_i(\psi)),
\end{align}
for all $\psi\in C^\infty(\mathcal{M})$.
Using the product rule, we compute that
\begin{align}
    \phi_j f_i(\phi_j f_i(\psi)) = \phi_j f_i (\phi_j) f_i(\psi) +  \phi_j^2 f_i(f_i(\psi)).
\end{align}
Therefore, for any $i\in\{1,\dots,r\}$, we have that
\begin{align}\label{eq:diff_defn_1}
    \begin{aligned}
        \sum_{j=1}^2\phi_j f_i(\phi_j f_i(\psi)) &= (\phi_1^2+ \phi_2^2) f_i(f_i(\psi))\\
        &+ (\phi_1 f_i (\phi_1)+
    \phi_2 f_i (\phi_2))f_i(\psi).
    \end{aligned}
\end{align}
From the chain rule and the definition of $F_j$, we compute that $f_i(\phi_1) = \phi_2 f_i(\phi)$, $f_i(\phi_2)= -\phi_1f_i(\phi)$, $\phi_1^2 + \phi_2^2 = 1$. Substituting into \eqref{eq:drift_defn_1} and \eqref{eq:diff_defn_1}, we obtain that
\begin{align*}
    f_0&= \bar{f}_0-\sum_{i=1}^r f_i(\phi)f_i, & \sum_{j=1}^2\phi_j f_i(\phi_j f_i(\psi))&= f_i(f_i(\psi)).
\end{align*}
Substituting into the infinitesmal generator, we obtain that $\mathcal{L}\psi= f_0(\psi) + \frac{1}{2\kappa}\sum_{i=1}^r f_i(f_i (\psi))$. If $\nabla$ is any affine connection on $\mathcal{M}$, then, using \cite[Proposition 4.15]{lee2018introduction}, we compute that
\begin{align}
    f_i(f_i (\psi)) = (\nabla_{f_i}d\psi)(f_i) + \nabla_{f_i}f_i (\psi), 
\end{align}
In particular, if $\nabla$ is the Levi-Civita connection associated with the metric \eqref{eq:metric}, then
\begin{align}
    \sum_{i=1}^r f_i(f_i(\psi)) &= \Delta  \psi + \sum_{i=1}^r \nabla_{f_i} f_i(\psi).
\end{align}
Therefore, the generator $\mathcal{L}$ acts on a smooth function $\psi$ as
\begin{align}
    \mathcal{L}\psi = f_0(\psi) + \frac{1}{2\kappa} \sum_{i=1}^r \nabla_{f_i} f_i(\psi) + \frac{1}{2\kappa}\Delta \psi,
\end{align}
so that, if the vector field $\bar{f}_0$ is taken to be
\begin{align}
    \bar{f}_0&= -\frac{1}{2\kappa} \sum_{i=1}^r \nabla_{f_i} f_i,
\end{align}
then the generator $\mathcal{L}$ is given by
\begin{align}
    \mathcal{L}\psi = -\langle\Grad\phi, \Grad\psi\rangle+ \frac{1}{2\kappa}\Delta \psi,
\end{align}
for any function $\psi\in C^\infty(\M)$. Due to compactness of $\mathcal{M}$ and the fact that the operator $\mathcal{L}$ is uniformly elliptic for any $\kappa > 0$, it follows that its adjoint $\mathcal{L}^*$, defined by $\mathcal{L}^* \psi = \Div(\psi \Grad\phi) + \frac{1}{2\kappa}\Delta  \psi$, has a unique invariant density. Indeed, the Fokker-Planck equation corresponding to \eqref{eq:SLBA-limit-sde-simplified} is
\begin{align*}
    \partial_t \rho &= \mathcal{L}^*\rho = \Div(\rho \Grad\phi) + \frac{1}{2\kappa}\Delta  \rho \\
    &= \frac{1}{2\kappa}\Div\left(\mathrm{e}^{-2\kappa \phi}\Grad\left(\mathrm{e}^{2\kappa \phi}\rho\right)\right).
\end{align*}
From the the last equality, we observe that the density $\rho_\kappa = c_\kappa^{-1}\,\mathrm{e}^{-2\kappa \phi}$,  $c_{\kappa}:=\int_{\mathcal{M}} \mathrm{e}^{-2\kappa \phi}\$ \volg$, where $\volg$ is the Riemannian density, is a stationary solution of the Fokker-Planck operator, and, therefore, $\mu_\kappa:= \rho_{\kappa} \,\volg$ is the unique invariant measure for \eqref{eq:sde}.  \hfill $\blacksquare$

\vspace{0.1cm}\noindent 
\textbf{Proof of Lemma \ref{lem:conc}:} Fix any open $U\supset S_\star$ and define the set $K:=\mathcal M\setminus U$, which is compact and disjoint from $S_\star$. By continuity of $\phi$ and compactness, the minimum $m_K:=\min_{\bar{x}\in K}\phi(\bar{x})$, exists and satisfies $m_K>\phi_\star$. Choose $\epsilon\in(0,m_K-\phi_\star)$ and define the nonempty closed set $N:=\{x\in\mathcal M:\ \phi(x)\le \phi_\star+\epsilon\}$. Then, we have that
\begin{align*}
\int_K e^{-2\kappa\phi}\,\volg \;\le\; e^{-2\kappa m_K}\,\volg(K),\\
\int_N e^{-2\kappa\phi}\,\volg \;\ge\; e^{-2\kappa(\phi_*+\epsilon)}\,\volg(N).
\end{align*}
Therefore, because $c_\kappa\ge \int_N e^{-2\kappa\phi}\,\volg$, we obtain that
\begin{align*}
\mu_\kappa(K)\le
\frac{e^{-2\kappa m_K}\,\volg(K)}{e^{-2\kappa(\phi_*+\varepsilon)}\,\volg(N)}
= C_K\,e^{-2\kappa c_K},
\end{align*}
where $C_K:=\mathrm{Vol}(K)\mathrm{Vol}(N)^{-1}$ and $c_K:=(m_K-\phi_*-\epsilon)>0$ are independent of $\kappa$.
Because $c_K>0$, the right-hand side of the inequality \eqref{eq:conc} decays exponentially as $\kappa\nearrow\infty$, hence $\mu_\kappa(K)\searrow 0$ and, therefore, $\mu_\kappa(U)\nearrow 1$. \hfill $\blacksquare$

\vspace{0.1cm}\noindent 
\textbf{Proof of Lemma \ref{lem:convergence}:} For any fixed $\kappa>0$, the operator $\mathcal{L}$ is elliptic with a spectral gap $\lambda_\kappa >0$. Consequently, the Markov semigroup $\mathrm{e}^{t\mathcal{L}}$ satisfies
    \begin{align*}
        \|\mathrm{e}^{t\mathcal{L}}f\|_{L^2(\mu_\kappa)}\leq \mathrm{e}^{-\lambda_\kappa t}\|f\|_{L^2(\mu_\kappa)},
    \end{align*}
    for any $f\in L^2(\mu_\kappa)$ with $\int_{\mathcal{M}} f \,\mu_\kappa = 0$. To proceed, define the function $u_t:= \rho_t\rho_\kappa^{-1}-1$,  so that $\rho_t=\rho_\kappa(u_t+1)$. Since, by definition, $\mathcal L^*\rho_\kappa=0$, and $\mathcal L^*(\rho_\kappa v)=\rho_\kappa\mathcal Lv$ for any smooth $v$, by self-adjointness of $\mathcal L$ in $L^2(\mu_\kappa)$, we obtain that
    $$
    \rho_\kappa\,\partial_t u_t=\partial_t\rho_t=\mathcal L^*\rho_t=\mathcal L^*(\rho_\kappa u_t)
    =\rho_\kappa\,\mathcal Lu_t,
    $$
    and hence $\partial_t u_t=\mathcal Lu_t$.
    Moreover, direct computation gives 
    $$\int_{\mathcal{M}} u_t \,\mu_\kappa = \int_{\mathcal{M}} (\rho_t-\rho_\kappa)\,\volg = 0,$$
    and, consequently, we obtain that $\|u_t\|_{L^2(\mu_\kappa)}\leq \mathrm{e}^{-\lambda_\kappa t }\|u_0\|_{L^2(\mu_\kappa)}$, for all $t\geq 0$. From the Cauchy-Schwartz inequality and the positivity of $\rho_{\kappa}$ for any $\kappa > 0$, we obtain that
    \begin{align*}
        \|\rho_t&-\rho_\kappa\|_{L^1(\volg)} = \int_{\mathcal{M}}|\rho_t-\rho_\kappa|\rho_\kappa^{-\frac{1}{2}}\rho_\kappa^{\frac{1}{2}}\,\volg \\
        &\leq \left(\int_{\mathcal{M}}(\rho_t-\rho_\kappa)^2\rho_\kappa^{-1}\,\volg\right)^{\frac{1}{2}}\left(\int_{\mathcal{M}}\rho_\kappa\,\volg\right)^{\frac{1}{2}}\\
        &= \left(\int_{\mathcal{M}}u_t^2\,\rho_\kappa\,\volg\right)^{\frac{1}{2}} = \|u_t\|_{L^2(\mu_\kappa)}\leq \mathrm{e}^{-\lambda_\kappa t }\|u_0\|_{L^2(\mu_\kappa)}.
    \end{align*}
    which proves the first claim. To show the second claim, we note that the operator $\mathcal{L}$ belongs to the class of operators considered in \cite[Section 8.2]{kolokoltsov2007semiclassical}. Therefore, it follows from \cite[Proposition 2.1]{kolokoltsov2007semiclassical}, in conjunction with the assumption that $\phi$ is a Morse function with a unique local minimum, that the spectral gap of $\mathcal{L}$ is uniformly bounded below for all $\kappa > 0$ by a strictly positive constant $\lambda>0$. \hfill $\blacksquare$

    \vspace{0.1cm}\noindent 
\textbf{Proof of Lemma \ref{lem:conv_sde_FK}:} By definition, the total variation distance between the probability laws $\mathbb{P}$ and $\mathbb{Q}$ is
    \begin{align}
        \sup_{A ~\text{measurable}} |\mathbb{P}(A)-\mathbb{Q}(A)|.
    \end{align}
    When the laws $\mathbb{P}$ and $\mathbb{Q}$ have densities $p$ and $q$, 
    it follows that
    \begin{align}
        \sup_{A ~\text{measurable}} |\mathbb{P}(A)-\mathbb{Q}(A)| = \frac{1}{2}\|p-q\|_{L^1(\volg)}.
    \end{align}
    Substituting for $\mathbb{P}$ with the law of the random variable $\bar{x}_t$ and for $\mathbb{Q}$ with the invariant measure $\mu_\kappa$, we obtain that
    \begin{align}
        |\mathbb{P}\{\bar{x}_t\in U\}-\mu_\kappa(U)| \leq \frac{1}{2}\|\rho_t-\rho_\kappa\|_{L^1(\volg)}, 
    \end{align}
    where $\rho_t$ solve the Fokker-Planck equation from the initial condition $\rho_0$. Hence, we obtain that
    \begin{align}
        |\mathbb{P}\{\bar{x}_t\in U\}-\mu_\kappa(U)| \leq \frac{1}{2}\|\rho_0-\rho_\kappa\|_{L^2(\rho_{\kappa}^{-1}\volg)}\mathrm{e}^{-\lambda_\kappa t},
    \end{align}
    for any measurable $U\subset\mathcal{M}$, any $\kappa>0$, and any $t\geq 0$. If, in addition, $\rho_0$ is the uniform probability density, then
    \begin{align*}
        \|\rho_0-\rho_\kappa&\|_{L^2(\rho_{\kappa}^{-1}\volg)}^2 = \int_{\mathcal{M}}(\rho_0-\rho_\kappa)^2\rho_\kappa^{-1}\volg\\
        &=\int_{\mathcal{M}} \rho_0^2 \,\rho_\kappa^{-1}\volg -2\int_{\mathcal{M}}\rho_0\,\volg + \int_{\mathcal{M}}\rho_\kappa\,\volg\\
        & = \frac{1}{\mathrm{Vol}(\mathcal{M})^2}\int_{\mathcal{M}} \rho_\kappa^{-1}\volg -1,
    \end{align*}
    where we used the fact that $\rho_0=\mathrm{Vol}(\mathcal{M})^{-1}$, $\int_{\mathcal{M}}\rho_\kappa\,\volg = 1$, and $\int_{\mathcal{M}}\rho_0\,\volg =1$. Recalling the form of the invariant measure, we see that
    \begin{align}
        \rho_\kappa (x) \geq \frac{\mathrm{e}^{-2\kappa \max_{x\in\mathcal{M}}\phi(x)}}{\mathrm{e}^{-2\kappa \min_{x\in\mathcal{M}}\phi(x)}\int_{\mathcal{M}} \,\volg}\geq \frac{\mathrm{e}^{-2\gamma_\phi\kappa}}{\mathrm{Vol}(\mathcal{M})},
    \end{align}
    for all $x\in\mathcal{M}$. Therefore, we obtain the bound
    \begin{align*}
        \|\rho_0-\rho_\kappa\|_{L^2(\rho_{\kappa}^{-1}\volg)}^2 &\leq\frac{\mathrm{e}^{2\gamma_\phi\kappa}}{\mathrm{Vol}(\mathcal{M})}\int_{\mathcal{M}}\frac{1}{\rho_\kappa}\volg -1 \\
        &= \mathrm{e}^{2\gamma_\phi\kappa}-1\leq\mathrm{e}^{2\gamma_\phi\kappa},
    \end{align*}
    from which the second inequality follows. \hfill $\blacksquare$

\vspace{0.1cm}\noindent 
\textbf{Proof of Proposition \ref{prop:finite_horizon}:} Using the triangle inequality, we have
    \begin{align*}
         |&\mathbb{P}\{x^\varepsilon_t\in U\}-\mu_{\kappa}(U)| \\
         &\leq  |\mathbb{P}\{x^\varepsilon_t\in U\}-\mathbb{P}\{\bar{x}_t\in U\}| +  |\mathbb{P}\{\bar{x}_t\in U\}-\mu_{\kappa}(U)| \\
         &\leq |\mathbb{P}\{x^\varepsilon_t\in U\}-\mathbb{P}\{\bar{x}_t\in U\}|+ \frac{1}{2}\|\rho_0-\rho_\kappa\|_{L^2(\rho_{\kappa}^{-1}\volg)}\mathrm{e}^{-\lambda_\kappa t}
    \end{align*}
    where we use Corollarly \ref{lem:conv_sde_FK}. Since $\rho_0\in C^\infty(\mathcal{M})$, the law of $\bar{x}_t$ admits a density for all $t\geq 0$ and, thus, any open set $U\subset \mathcal{M}$ with a smooth boundary is a continuity set of the law of $\bar{x}_t$. Consequently, from Theorem \ref{thm:scot}, we obtain that there exists $\varepsilon^*>0$ such that, for all $\varepsilon\in(0,\varepsilon^*)$ and all $t\in[0,T]$,
    \begin{align}
        |\mathbb{P}\{x^\varepsilon_t\in U\}-\mathbb{P}\{\bar{x}_t\in U\}| < \epsilon,
    \end{align}
    which concludes the proof. \hfill $\blacksquare$   

\vspace{0.1cm}\noindent
\textbf{Proof of Lemma \ref{lem:markov}:} On $[k\varepsilon,(k+1)\varepsilon)$, differentiating \eqref{eq:B} in $t$ eliminates the additive constant $W^\ell_{k\varepsilon}$, so $b^\ell_\varepsilon(t,W)$ depends on $W$ only through the increments ${w}_k$ which, by definition, are i.i.d samples from the Gaussian distribution $\mathcal{N}(0,\varepsilon I)$. Therefore, for any fixed $w\in\mathbb{R}^{2r}$, if we define $\Psi_\varepsilon(x,{w})$ as the time-$1$ solution of the ODE
\begin{equation}\label{eq:onestep}
y'(s)=\varepsilon\bar f_0(y)+\tfrac{1}{\sqrt\kappa}\sum_{i=1}^r\sum_{j=1}^2 g_{i,j}(y)\,{w}^{2i+j-2}\,(U^{\iota_i^j({w})}_\kappa)'(s),
\end{equation}
from the initial condition $y(0)=x$, where the vector fields $g_{i,j}$ are defined by $g_{i,j}:=F_j\circ\phi\,f_i$, we see that $\Psi_\varepsilon(x,{w})$ is independent of $k$, and, therefore, the sampled process $X_k$ is a time-homogeneous Markov process, where the Markov property follows from the independence of the increments $w_k$. To complete the proof, it remains to show that $P_\varepsilon$ maps $C(\M)$ into $C(\M)$. Since ${w}^{2i-1}{w}^{2i}=0$ if and only if one of the two coordinates vanishes, the set $
\mathcal{D}_0:=\bigcup_i^{r}\{{w}\in\mathbb{R}^{2r} : {w}^{2i-1}{w}^{2i}=0\}
$ coincides with the union of the coordinate hyperplanes $\{{w}\in\mathbb{R}^{2r}~|~{w}^\ell=0\}$. In particular, $\mathcal{D}_0$ is closed and $\mathrm{Leb}(\mathcal{D}_0) = 0$, which implies that it also has zero measure with respect to the Gaussian measure $\mathcal{N}(0,\varepsilon I)$ since the latter is absolutely continuous with respect to the former. Its complement is the disjoint union of the $4^r$ open orthants of $\mathbb{R}^{2r}$. On each orthant, the sign pattern of ${w}$, and therefore the branch $\iota_i^j({w})$ of every pair in \eqref{eq:onestep} is constant. Consequently, for ${w}$ ranging over a fixed orthant, the right-hand side of \eqref{eq:onestep} is a single smooth function of $(y,{w},s)$. By smooth dependence of ODE solutions on initial conditions and parameters, the map $\Psi_\varepsilon$ is jointly smooth on $\M\times O$ for each orthant $O\subset\mathbb{R}^{2r}$. Borel measurability of $\Psi_\varepsilon$ on all of $\M\times\mathbb{R}^{2r}$ follows, since its domain partitions into finitely many Borel sets, on each of which it is continuous.
 
Nevertheless, continuity of $\Psi_\varepsilon(x,\cdot)$ on all of $\mathbb{R}^{2r}$ fails. Indeed, the map $\Psi_\varepsilon(x,\cdot)$ has a genuine discontinuity across the set $\mathcal{D}_0$. However, because the discontinuity set $\mathcal{D}_0$ has zero measure, we are allowed to proceed as follows: 
Fix a function $\psi\in C(\M)$, a sequence $x_n\to x$ in $\M$, and define 
\begin{align*}
h_n({w})&:=\psi(\Psi_\varepsilon(x_n,{w})), & h({w})&:=\psi(\Psi_\varepsilon(x,{w})).
\end{align*} 
For every ${w}\notin\mathcal{D}_0$, the point ${w}$ lies in some orthant $O$, and $(x_n,{w})\to(x,{w})$ within $\M\times O$, on which $\Psi_\varepsilon$ is continuous. Hence, $h_n({w})\to h({w})$, and $h_n\to h$ pointwise $\mathcal{N}(0,\varepsilon I)$-almost everywhere. Since, by definition, $|h_n|\le\|\psi\|_\infty$, the Dominated Convergence theorem \cite[Theorem 3.3.1]{axler2020measure} yields
$$
P_\varepsilon\psi(x_n)=\int h_n\,\mathrm{d}\gamma_\varepsilon\to\int h\,\mathrm{d}\gamma_\varepsilon=P_\varepsilon\psi(x),
$$
i.e., $P_\varepsilon\psi$ is sequentially continuous whenever $\psi\in C(\M)$. Because $\M$ is metrizable, sequential continuity implies continuity, and we have $P_\varepsilon\psi\in C(\M)$. \hfill $\blacksquare$

\vspace{0.1cm}\noindent 
\textbf{Proof of Proposition \ref{prop:existence}:} From \cite[Prop.~6.1.1]{meyn2012markov}, the chain $(X_k)$ is weak Feller since its kernel $P_\varepsilon$ satisfies $P_\varepsilon\psi\in C(\M)$ for any $\psi\in C(\M)$.  From \cite[Prop.~12.0.1]{meyn2012markov}, it follows that the chain $(X_k)$ admits at least one invariant measure. \hfill $\blacksquare$

\vspace{0.1cm}\noindent 
\textbf{Proof of Theorem \ref{thm:lawconv}:} Since \eqref{eq:sde} is uniformly
elliptic with smooth coefficients on the compact $\M$, its transition kernel admits a density $p_t(x,y)$,
jointly smooth and strictly positive for $t>0$. Moreover, for any $\kappa>0$, the corresponding operator $\mathcal{L}$ has a strictly positive spectral gap $\lambda_\kappa>0$. Define the constant
$$
C_\kappa:=\sup_{x\in\M}\|p_1(x,\cdot)-\rho_\kappa\|_{L^2(\rho_\kappa^{-1}\volg)},
$$
which is finite by joint continuity
of $p_1$ on $\M\times\M$ and the positivity of $\rho_\kappa$. For $t\ge1$, the Markov property identifies the
density $\rho^x_t$ of $\bar x_t(x)$ with the solution of the Fokker--Planck equation at time $t-1$ from the
initial density $p_1(x,\cdot)\in C^\infty(\M)$, so Proposition~\ref{lem:convergence} yields
$$
\|\rho^x_t-\rho_\kappa\|_{L^1(\volg)}\le C_\kappa e^{-\lambda_\kappa(t-1)},
$$ 
uniformly in $x$. Let
$T:=1+\lambda_\kappa^{-1}\log(2C_\kappa\delta^{-1})$. Then, for every $t\ge T$ and every $f$ with $\|f\|_\infty\le1$,
$$
\sup_{x\in\M}\big|\mathbb Ef(\bar x_t(x))-\mu_\kappa(f)\big|
\ \le\ \sup_{x\in\M}\|\rho^x_t-\rho_\kappa\|_{L^1(\volg)}\ \le\ \frac{\delta}{2}.
$$
Next, let
$\varepsilon^\ast:=\varepsilon^\ast\big((\frac{\delta}{2})^2,\,T+1\big)$ be given by Theorem \ref{thm:scot}, fix $\varepsilon\in(0,\varepsilon^\ast]$, and
set $k_\varepsilon:=\lceil T\varepsilon^{-1}\rceil$, so that $t_\varepsilon:=k_\varepsilon\varepsilon\in[T,T+1]$
and $P^{k_\varepsilon}_\varepsilon(x,\cdot)$ is the law of $x^\varepsilon_{t_\varepsilon}(x)$. Then, the convergence of trajectories property \eqref{eq:wz} along with Jensen's
inequality give
$$
\big|P^{k_\varepsilon}_\varepsilon f(x)-\mathbb Ef(\bar x_{t_\varepsilon}(x))\big|
\ \le\ \mathbb E\big|x^\varepsilon_{t_\varepsilon}(x)-\bar x_{t_\varepsilon}(x)\big| \le \frac{\delta}{2},
$$
for $f$ with $\|f\|_\infty\le1$, $\mathrm{Lip}(f)\le1$, uniformly in $x$, so that 
\begin{equation}\label{eq:dbl_bound}
    \sup_x d_{\mathrm{BL}}\big(P^{k_\varepsilon}_\varepsilon(x,\cdot),\mu_\kappa\big)\le\delta
\end{equation}
Next, let $k\in\mathbb N$ with $k\ge T\varepsilon^{-1}$, so that $k\ge k_\varepsilon$. Because
$P^{k-k_\varepsilon}_\varepsilon(x,\cdot)$ is a probability measure, the Chapman--Kolmogorov relation allows us to write
\begin{align*}
    P^{k}_\varepsilon f(x)-\mu_\kappa(f)
    =\int_\M\big(P^{k_\varepsilon}_\varepsilon f(y)-\mu_\kappa(f)\big)\,P^{k-k_\varepsilon}_\varepsilon(x,\mathrm dy),
\end{align*}
and, since the bound \eqref{eq:dbl_bound} is uniform in the base point, the right-hand side is bounded in
absolute value by $\delta$. Taking the supremum over all $f$ with $\|f\|_\infty\le1$ and $\mathrm{Lip}(f)\le1$
yields the assertion. Finally, if $\mu_\varepsilon P_\varepsilon=\mu_\varepsilon$, then
$\mu_\varepsilon=\mu_\varepsilon P^{k_\varepsilon}_\varepsilon$, so that
\begin{align*}
    \mu_\varepsilon(f)-\mu_\kappa(f)
    =\int_\M\big(P^{k_\varepsilon}_\varepsilon f(x)-\mu_\kappa(f)\big)\,\mu_\varepsilon(\mathrm dx),
\end{align*}
for any such $f$, which is again bounded by $\delta$ in absolute value, and, therefore,
$d_{\mathrm{BL}}(\mu_\varepsilon,\mu_\kappa)\le\delta$. \hfill $\blacksquare$

\vspace{0.1cm}\noindent 
\textbf{Proof of Proposition \ref{prop:practical}:} We first claim that, for every Borel set $A\subseteq\M$ and every $\eta\in(0,1]$, we have that
\begin{align}
    \inf_{x\in\M}P^{k}_\varepsilon\big(x,A^\eta\big)&\ \ge\ \mu_\kappa(A)-\delta\,\eta^{-1},\label{eq:claim_lower}\\
    \sup_{x\in\M}P^{k}_\varepsilon\big(x,A\big)&\ \le\ \mu_\kappa(A^\eta)+\delta\,\eta^{-1}.\label{eq:claim_upper}
\end{align}
Indeed, the function $f:=\max\{0,\,1-\eta^{-1}\dist(\cdot,A)\}$ satisfies
$$
\mathbf 1_{A}\ \le\ f\ \le\ \mathbf 1_{A^\eta},\qquad \|f\|_\infty\le1,\qquad \mathrm{Lip}(f)\le\eta^{-1}.
$$
Since $\eta\le1$, the rescaled function $\eta f$ satisfies $\|\eta f\|_\infty\le1$ and
$\mathrm{Lip}(\eta f)\le1$, so applying Theorem~\ref{thm:lawconv} to $\eta f$ and dividing by $\eta$, using
linearity of $P^{k}_\varepsilon$ and of $\mu_\kappa$, yields
$$
\big|P^{k}_\varepsilon f(x)-\mu_\kappa(f)\big|\ \le\ \delta\,\eta^{-1},\qquad\forall x\in\M.
$$
The claim then follows from the two chains of inequalities
$$
P^{k}_\varepsilon\big(x,A^\eta\big)\ \ge\ P^{k}_\varepsilon f(x)\ \ge\ \mu_\kappa(f)-\delta\,\eta^{-1}
\ \ge\ \mu_\kappa(A)-\delta\,\eta^{-1},
$$
$$
P^{k}_\varepsilon\big(x,A\big)\ \le\ P^{k}_\varepsilon f(x)\ \le\ \mu_\kappa(f)+\delta\,\eta^{-1}
\ \le\ \mu_\kappa(A^\eta)+\delta\,\eta^{-1}.
$$
To prove the first claim, note that $\eta_U>0$, since $S_\star$ is compact, $\M\setminus U$ is closed and nonempty,
and the two sets are disjoint. We now apply \eqref{eq:claim_lower} with $A=S_\star^{\eta_U}$ and $\eta:=\eta_U$. By the triangle inequality, we have that
$(S_\star^{\eta_U})^{\eta_U}\subseteq S_\star^{2\eta_U}$. Also, since $2\eta_U\le\dist(S_\star,\M\setminus U)$, we have that $S_\star^{2\eta_U}\subseteq U$. It follows that
$$
P^{k}_\varepsilon\big(x,U\big)\ \ge\ P^{k}_\varepsilon\big(x,(S_\star^{\eta_U})^{\eta_U}\big)
\ \ge\ \mu_\kappa\big(S_\star^{\eta_U}\big)-\delta\,\eta_U^{-1},
$$
and, since the set $\M\setminus S_\star^{\eta_U}$ is closed and disjoint from $S_\star$, we obtain from \eqref{eq:conc} that 
$$
\mu_\kappa(S_\star^{\eta_U})\ge1-C_U\,e^{-2\kappa c_U},
$$
which proves the first claim.

To prove the second claim, note that $\eta_K>0$, since $K$ is compact, $S_\star$ is closed, and the two sets are disjoint. We apply \eqref{eq:claim_upper} with $A:=K$ and $\eta:=\eta_K$, so that
$$
P^{k}_\varepsilon\big(x,K\big)\ \le\ \mu_\kappa\big(K^{\eta_K}\big)+\delta\,\eta_K^{-1}
\ \le\ \mu_\kappa\big(\overline{K^{\eta_K}}\big)+\delta\,\eta_K^{-1}.
$$
By definition, every $y\in\overline{K^{\eta_K}}$ satisfies $\dist(y,K)\le\eta_K$, while, by construction $\dist(K,S_\star)\ge2\eta_K$. Hence, the set $\overline{K^{\eta_K}}$ is closed and disjoint from $S_\star$. We then obtain from \eqref{eq:conc} that
$$
\mu_\kappa\big(\overline{K^{\eta_K}}\big)\le C_K\,e^{-2\kappa c_K},
$$
which proves the second claim and concludes the proof. \hfill $\blacksquare$
\section{NUMERICAL EXAMPLE}
\label{sec:example}
%
%\begin{example}\label{exmp:sphere}\normalfont
\begin{figure*}[t]
    \centering
        \hfill
        \includegraphics[width=\textwidth]{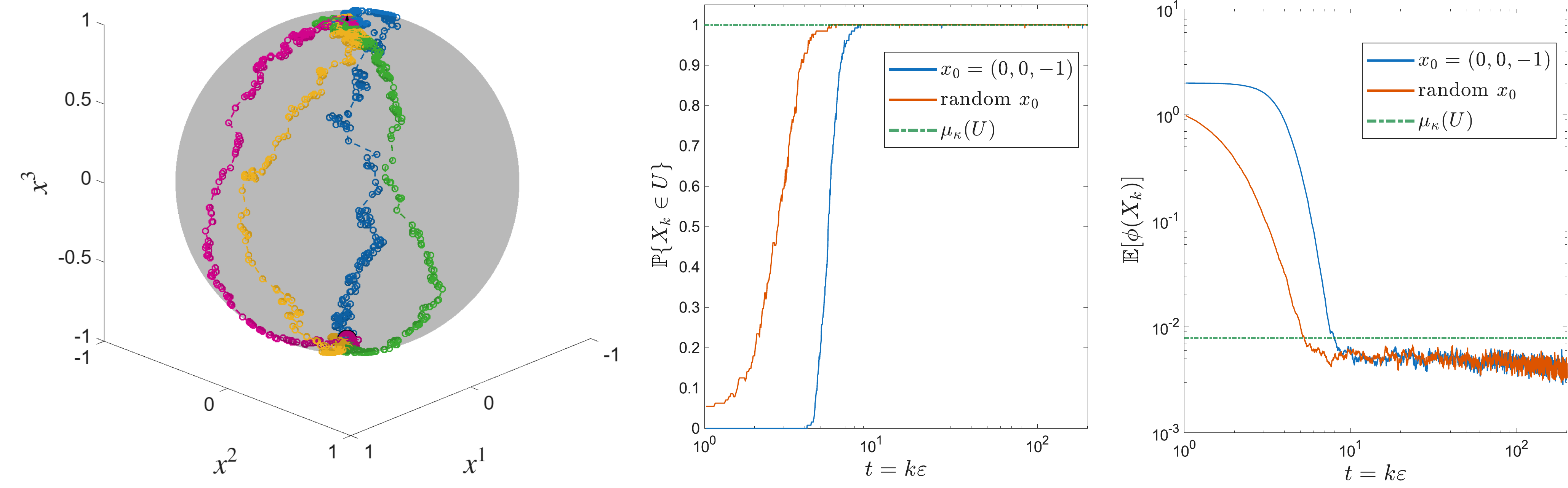}
        \caption{(Left) the sample paths generated by the random ODE \eqref{eq:rode}, (center) empirical approximations of $\mathbb{P}\{X_k\in U\}$, and (right)  empirical approximations of $\mathbb{E}[\phi(X_k)]$.}
        \label{fig:phi_sphere}
        \vspace{-0.2cm}
\end{figure*}
Let $\mathcal{M}=\mathbb{S}^2$ and $\mathcal{F} = \{f_1,f_2,f_3\}$ where the vector fields $f_i$ are given by $f_1(x)=x^2\partial_{1}-x^1\partial_{2}$, $f_2(x)= x^3\partial_2-x^2\partial_3$, and $f_3(x)= x^1\partial_3-x^3\partial_1$ in the ambient coordinates on $\mathbb{R}^3$. Fix any $x_\star\in\mathbb{S}^2$ and let $\phi$ be given by
    \begin{align*}
        \phi(x):=1-(x_\star^1x^1 + x_\star^2x^2+x_\star^3x^3).
    \end{align*}
    The function $\phi$ is a Morse function. Moreover, $\phi$ has two critical points, namely, $\pm x_\star$, and only the critical point $x_\star$ is a local minimum. The Riemannian metric $\langle \cdot,\cdot\rangle$ corresponding to the family $\mathcal{F}$ is defined in the ambient coordinates by
    \begin{align}
        \langle \cdot,\cdot\rangle &= \sum_{i,j=1}^3g_{ij}\mathrm{d}x^i\otimes \mathrm{d}x^j, &
        G(x)&= [g_{ij}]= I - x x^\top.
    \end{align}
    In particular, the metric $\langle \cdot,\cdot\rangle$ coincides with the induced metric inherited from the ambient Euclidean metric on $\mathbb{R}^3$. Consequently, if $D_XY$ denotes the directional derivative of the vector field $Y$ along the vector field $X$ in the ambient space $\mathbb{R}^3$, then the Levi-Civita connection on $\mathcal{M}$ in the ambient coordinates is given by $\nabla_X Y = G(x)^\dagger D_XY$, where $G(x)^\dagger$ is the Moore-Penrose inverse of $G(x)$. Utilizing this expression, we compute that $\bar{f}_0 = 0$. We now generate several sample paths of the random ODE \eqref{eq:rode} from the initial condition $-x_\star$. For simplicity, we fix $x_\star = (0,0,1)$. A plot of some sample paths of the chain $X_k$ on $\mathcal{M}$ are shown in Figure \ref{fig:phi_sphere} for $\varepsilon = 0.02$, and $\kappa = 25$. As expected, the sample paths converge to the vicinity global minimizer of the function $\phi$. Moreover, the trajectories remain on $\mathcal{M}$ for all time. We also plot ensemble statistics by simulating $N=256$ independent sample paths and computing the empirical approximations for the probability $\mathbb{P}\{X_k\in U\}$, where $U:=\{x\in\mathcal{M}~|~\phi(x)<0.06\}$, as well as the expectation $\mathbb{E}[\phi(X_k)]$. As shown in Figure \ref{fig:phi_sphere}, the approximations converge to reasonable margins of the corresponding values for the associated Langevin diffusion, which corroborates our analysis.

\vspace{-0.2cm}
\section{CONCLUSIONS}
\label{sec:conclusions}
We introduced and analyzed a novel stochastic Lie-bracket approximation algorithm for global zeroth-order optimization on compact manifolds without boundary. Our analysis shows that the proposed algorithm leads to finite-horizon, as well as infinite-horizon, practical probabilistic guarantees on the convergence of the proposed scheme. Numerical results are provided to illustrate the behavior of the algorithm. Future work will consider extensions of the proposed algorithm to compact manifolds with boundaries. 

\vspace{-0.2cm}
\bibliography{biblio}
\bibliographystyle{ieeetr}

\end{document}